\documentclass[a4j,12pt,final]{article}
\usepackage{amsmath}
\usepackage{amssymb}
\usepackage{amsthm}
\usepackage{amscd}
\newtheorem{theorem}{Theorem}[section]
\newtheorem{lemma}{Lemma}[section]
\theoremstyle{definition}
\newtheorem{definition}{Definition}[section]
\theoremstyle{remark}
\newtheorem{remark}{Remark}[section]
\theoremstyle{proposition}
\newtheorem{proposition}{Proposition}[section]
\numberwithin{equation}{section}

\theoremstyle{corollary}

\theoremstyle{errata}
\newtheorem{errata}{Errata}

\begin{document}

\title{On some cocycles which represent the Dixmier-Douady class in simplicial de Rham complexes}
\author{ Naoya Suzuki}
\date{}
\maketitle
\begin{abstract}
When a Lie group $G$ has a central $U(1)$-extension, there is a cocycle in the simplicial de Rham complex $\Omega^3(NG)$ which represents the Dixmier-Douady class. Mickelsson and Brylinski, McLaughlin constructed a central $U(1)$-extension $\widehat{LSU(2)} \rightarrow LSU(2)$ whose Dixmier-Douady class in $\Omega^3(NLSU(2))$ is a
kind of transgression of the second
Chern class.

In this paper, we consider the case of unitary group and construct a central $U(1)$-extension of $LU(2)$.
After that we construct also a cocycle in a certain triple complex.
\end{abstract}

\section{Introduction}
It is well known that for any Lie group $G$, we can define a simplicial manifold $\{ NG(*) \}$ and we can see the cohomology group of the classifying space $BG$ is isomorphic to the total cohomology 
of the double complex $  {\Omega}^{q} (NG(p)) $.

On the other hand, in \cite{Car} Carey, Crowley, Murray proved that when  a Lie group $G$ admits a central $U(1)$-extension $1 \rightarrow U(1) \rightarrow \widehat{G} \rightarrow G \rightarrow 1$, there exists a  characteristic class of principal $G$-bundle $\pi :Y \rightarrow M $ which belongs to a cohomology group $H^3 (M , \mathbb{Z} )$. It is called the Dixmier-Douady (DD) class associated to the
central $U(1)$-extension $\widehat{G} \rightarrow G$. 
So there is a cocycle on  ${\Omega}^{*} (NG(*)) $ which represents the DD class. 

Mickelsson \cite{Mic} and Brylinski, McLaughlin \cite{Bry}\cite{Bry-ML} constructed a $U(1)$-central extension $\widehat{LSU(2)} \rightarrow LSU(2)$ and a connection form on it whose DD class in $\Omega^3(NLSU(2))$ is a
kind of transgression of the second
Chern class.

In this paper, on the basis of Murray and Stevenson's idea\cite{Mur} \cite{Mur2}, we consider the case of unitary group and construct a central $U(1)$-extension of $LU(2)$.

We consider also the case of a semi-direct product $LSU(2) \rtimes S^1$ and construct a cocycle in a certain triple complex.

\section{Dixmier-Douady class on the double complex }

In this section we recall the relation between the simplicial manifold $NG$ and the classifying space $BG$, then we show that
we can construct a cocycle on ${\Omega}^{*} (NG(*)) $ which represents the DD class when $G$ has a central $U(1)$-extension $\pi : \widehat{G} \rightarrow G $.

\subsection{The double complex on simplicial manifold}

We define a simplicial manifold $NG$ for a Lie group $G$ as follows:\\
\par
$NG(p)  = \overbrace{G \times \cdots \times G }^{p-times}  \ni (g_1 , \cdots , g_p ) :$  \\
face operators \enspace ${\varepsilon}_{i} : NG(p) \rightarrow NG(p-1)  $
$$
{\varepsilon}_{i}(g_1 , \cdots , g_p )=\begin{cases}
(g_2 , \cdots , g_p )  &  i=0 \\
(g_1 , \cdots ,g_i g_{i+1} , \cdots , g_p )  &  i=1 , \cdots , p-1 \\
(g_1 , \cdots , g_{p-1} )  &  i=p.
\end{cases}
$$

Then we recall how to construct a double complex associated to a simplicial manifold.

\begin{definition}
For any simplicial manifold $ \{ X_* \}$ with face operators $\{ {\varepsilon}_* \}$, we define a double complex as follows:
$${\Omega}^{p,q} (X) := {\Omega}^{q} (X_p). $$
Derivatives are:
$$ d' := \sum _{i=0} ^{p+1} (-1)^{i} {\varepsilon}_{i} ^{*}  , \qquad  d'' := (-1)^{p} \times {\rm the \enspace exterior \enspace differential \enspace on \enspace }{ \Omega ^*(X_p) }. $$
\hspace{31em} $ \Box $ \\
\end{definition}

For $NG$ the following holds \cite{Bot2}\cite{Dup2}\cite{Mos}.

\begin{theorem}
{ There exists a ring isomorphism }
$$ H({\Omega}^{*} (NG))  \cong  H^{*} (BG ).$$
{ Here} ${\Omega}^{*} (NG)$ { means the total complex and }$BG${ means the classifying space of principal}
$G${-bundles}.\hspace{21em} $ \Box $ 
\end{theorem}

\subsection{The cocycle on the double complex }

Let $\pi : \widehat{G} \rightarrow G $ be a central $U(1)$-extension of a Lie group $G$. Following \cite{Bry-ML} \cite{Car}, we recognize it as a $U(1)$-bundle. Using the face operators $\{ {\varepsilon}_{i} \}: NG(2) \rightarrow NG(1) =G $, 
we can construct a $U(1)$-bundle over $NG(2)= G \times G $ as $\delta \widehat{G} :=  {\varepsilon _0}^* \widehat{G} \otimes ({{\varepsilon} _1}^* \widehat{G})^{{\otimes}-1} \otimes
{{\varepsilon} _2}^* \widehat{G} $.
Here the tensor product $S \otimes T$ of $U(1)$-bundles $S$ and $T$ over $M$ is defined as:
$$ S \otimes T := \bigcup _{x \in M} (S_x \times T_x) / (s,t) \sim (su,tu^{-1}),~~ (u \in U(1)).$$

\begin{lemma}
 $\delta \widehat{G} \rightarrow G \times G $ { is a trivial bundle.}
\end{lemma}
\begin{proof}
See \cite{Suz}.
\end{proof}

\begin{remark}
$\delta(\delta \widehat{G})$ is
canonically isomorphic to $G \times G \times G \times U(1)$ because ${\varepsilon}_{i}{\varepsilon}_{j} ={\varepsilon}_{j-1}{\varepsilon}_{i}$ for $i<j$.
\end{remark}
For any connection $\theta$ on $\widehat{G}$, there is the induced connection $\delta \theta $ on $\delta \widehat{G} $ \cite[Brylinski]{Bry}.

\begin{proposition}
{Let} $c_1 (\theta )$ denote the 2-form on $G$ which hits $\left( \frac{-1}{2 \pi i} \right) d \theta \in \Omega ^2 (\widehat{G})$ by $\pi ^{*}$, and $\hat{s}$ a global section of
$\delta \widehat{G} $ such that $\delta \hat{s}:= {\varepsilon}_{0} ^* \hat{s} \otimes ({\varepsilon}_{1} ^* \hat{s})^{\otimes -1} \otimes {\varepsilon}_{2} ^* \hat{s} \otimes ({\varepsilon}_{3} ^* \hat{s})^{\otimes -1}=1$. Then the following equations hold.
$$ ({\varepsilon}_{0} ^* - {\varepsilon}_{1} ^* + {\varepsilon}_{2} ^* ) c_1 (\theta ) = \left( \frac{-1}{2 \pi i} \right)d(\hat{s}^{*} (\delta \theta )) \enspace \in \Omega ^2 (NG(2)) .$$ 
$$({\varepsilon}_{0} ^* - {\varepsilon}_{1} ^* + {\varepsilon}_{2} ^* - {\varepsilon}_{3} ^*)(\hat{s}^{*} (\delta \theta ))=0.$$
\end{proposition}

\begin{proof}
See \cite{Mur}\cite{Mur2} or \cite{Suz}.
\end{proof}

\par

\rm{The propositions above give the cocycle} $c_1 (\theta ) - \left( \frac{-1}{2 \pi i} \right)\hat{s}^{*} (\delta \theta ) \in \Omega ^3 (NG)$ below.

$$
\begin{CD}
0 \\
@AA{-d}A \\
c_1 ( \theta ) \in {\Omega}^{2} (G )@>{{\varepsilon}_{0} ^* - {\varepsilon}_{1} ^* + {\varepsilon}_{2} ^* }>>{\Omega}^{2} (G \times G)\\
@.@AA{d}A\\
@.-\left( \frac{-1}{2 \pi i} \right) \hat{s}^{*} (\delta \theta ) \in {\Omega}^{1} (G \times G)@>{{\varepsilon}_{0} ^* - {\varepsilon}_{1} ^* + {\varepsilon}_{2} ^* - {\varepsilon}_{3} ^*}>> 0
\end{CD}
$$

\begin{lemma}
{The cohomology class} $[c_1 (\theta ) - \left( \frac{-1}{2 \pi i} \right)\hat{s}^{*} (\delta \theta )] \in H^3 (\Omega  (NG))$ does not depend on $\theta$.
\end{lemma}

\begin{proof}
See \cite{Suz}.
\end{proof}
\begin{errata}
{In} \cite{Suz}, the author insisted that $\hat{s}$ can be any section, but we must take $\hat{s}$ as above. Besides, the author
apologize that there are some confusions of the signature on the diagram of the double complex in  \cite{Suz}. The statement of Theorem 4.2 in it also must be
modified to ``The transgression map $H^*(BG, *) \rightarrow H^*(EG,G) \rightarrow H^{*-1}(G)$ of the universal bundle  $EG \rightarrow BG$ maps the Dixmier-Douady class to $(-1)$ times of the first Chern class of $\widehat{G} \rightarrow G$." \hspace{30em} $ \Box $ 
\end{errata}

Now we consider what happens if we change the section $\hat{s}$. There is a natural section $\hat{s}_{nt}$ of $\delta \widehat{G} $
defined as:
$$\hat{s}_{nt}(g_1,g_2):=[((g_1,g_2),\hat{g}_2 ) ,  ((g_1,g_2),\hat{g} _1 \hat{g} _2  )^{\otimes -1},((g_1,g_2),\hat{g}_1 ) ].$$
Then any other section $\hat{s}$ such that ${\delta}\hat{s}=1$ can be represented by $\hat{s} = \hat{s}_{nt} \cdot {\varphi}$
where ${\varphi}$ is a $U(1)$-valued smooth function on $G \times G$ which satisfies $\delta {\varphi}=1$.
If we pull back $\delta \theta$ by $\hat{s}$, the equation $\hat{s}^*(\delta \theta) = \hat{s}^* _{nt}(\delta \theta) + d {\rm{log}} {\varphi}$
holds. If there exists a $U(1)$-valued smooth function ${\varphi}'$ on $G $ which satisfies $\delta {\varphi}'={\varphi}$, the
cohomology class $[- \left( \frac{-1}{2 \pi i} \right) d {\rm{log}} {\varphi}]$ is equal to $0$ in $H^3 (\Omega  (NG))$.
So we have the following proposition.
\begin{proposition}
{Up to the cohomology class in the smooth cohomology } $H^2(G, {U(1)})$, the cohomology class $[c_1 (\theta ) - \left( \frac{-1}{2 \pi i} \right)\hat{s}^{*} (\delta \theta )] \in H^3 (\Omega  (NG))$ is decided uniquely by the central $U(1)$-extension $\widehat{G} \rightarrow G$.\hspace{8em} $ \Box $
\end{proposition}

\subsection{Dixmier-Douady class}

We recall the definition of the Dixmier-Douady class, following \cite{Car}.
Let $\phi : Y \rightarrow M $ be a principal $G$-bundle  and  $\{ U_{\alpha } \}$ a Leray covering of $M$.
When $G$ has a central $U(1)$-extension $\pi : \widehat{G} \rightarrow G $,
  the transition functions $ g _{\alpha \beta } : U_{\alpha \beta } \rightarrow G$ lift to $ \widehat{G} $. i.e.
there exist continuous maps $\hat{g} _{\alpha \beta } : U_{\alpha \beta } \rightarrow \widehat{G} $ such that
$\pi \circ \hat{g}_{\alpha \beta } = g _{\alpha \beta }$. This is because each $U_{\alpha \beta }$ is contractible so
the pull-back of $\pi $ by $ g _{\alpha \beta }$ has a global section.
Now the $U(1)$-valued functions $c_{\alpha \beta \gamma}$ on $U_{\alpha \beta \gamma}$ are defined
as $(\hat{g} _{ \beta \gamma }  (\hat{g} _{\alpha \beta   }  \hat{g}_{\beta \gamma})^{-1} \hat{g} _{\alpha \beta }) \cdot c_{\alpha \beta \gamma}:= \hat{g} _{ \beta \gamma }  \hat{g}^{-1} _{\alpha \gamma   } \hat{g} _{\alpha \beta } \in g_{ \beta \gamma } ^*  \widehat{G} \otimes ( g_{ \alpha \gamma } ^*  \widehat{G})^{\otimes -1}
\otimes g_{\alpha \beta  } ^*  \widehat{G}$. 
Then it is easily seen that $\{ c_{\alpha \beta \gamma} \}$ is a $U(1)$-valued \v{C}ech-cocycle on $M$, hence 
it defines a cohomology class in $H^2 (M , \underline{U(1)} ) \cong H^3 (M , \mathbb{Z} )$.
 This class is called the Dixmier-Douady class of $Y$. 
\begin{remark}
Let $s_{\alpha \beta \gamma}$ be a section of $\widehat{G}_{\alpha \beta \gamma} := g_{ \beta \gamma } ^*  \widehat{G} \otimes ( g_{ \alpha \gamma } ^*  \widehat{G})^{\otimes -1}
\otimes g_{\alpha \beta  } ^*  \widehat{G}$ such that $\delta s_{\alpha \beta \gamma} := s_{ \beta \gamma \delta} \otimes s_{\alpha \gamma \delta} ^{\otimes -1} \otimes
s_{\alpha \beta \delta} \otimes s_{\alpha \beta \gamma} ^{\otimes -1} = 1$. This condition makes sense since  $\widehat{G}_{ \beta \gamma \delta} \otimes \widehat{G}_{\alpha \gamma \delta} ^{\otimes -1} \otimes
\widehat{G}_{\alpha \beta \delta} \otimes \widehat{G}_{\alpha \beta \gamma} ^{\otimes -1}$ is canonically trivial.
Then we can define a $U(1)$-valued \v{C}ech-cocycle $c_{\alpha \beta \gamma} ^s$ on $M$ by the equation $s_{\alpha \beta \gamma} \cdot c_{\alpha \beta \gamma} ^s =\hat{g} _{ \beta \gamma }  \hat{g}^{-1} _{\alpha \gamma   } \hat{g} _{\alpha \beta }$. The cohomology class $[c_{\alpha \beta \gamma} ^s] \in H^2 (M , \underline{U(1)} ) \cong H^3 (M , \mathbb{Z} )$ can be also called the Dixmier-Douady class of $Y$.

\end{remark}

We fix any section $\hat{s}$ of $\delta \widehat{G} $ which satisfies $\delta s=1$. Since $ g_{ \beta \gamma } ^*  \widehat{G} \otimes ( g_{ \alpha \gamma } ^*  \widehat{G})^{\otimes -1} \otimes g_{\alpha \beta  } ^*  \widehat{G}$ is the pull-back of $\delta \widehat{G}$
by $ ( g_{\alpha \beta  } , g_{ \beta \gamma } ) : U_{\alpha \beta \gamma} \rightarrow G \times G $,  there is the induced section of $ g_{ \beta \gamma } ^*  \widehat{G} \otimes ( g_{ \alpha \gamma } ^*  \widehat{G})^{\otimes -1}
\otimes g_{\alpha \beta  } ^*  \widehat{G} $. So we can define the DD class
by using this section. \\

In \cite{Suz}, the theorem below was shown.

\begin{theorem}
{The cohomology class} $[c_1 (\theta ) - \left( \frac{-1}{2 \pi i} \right)\hat{s}^{*} (\delta \theta )] \in H^3 (\Omega  (NG))$
represents the universal Dixmier-Douady class associated to $\pi$ and a section $\hat{s}$.\hspace{30em} $ \Box $

\end{theorem}

\section{Another description of the DD class}

There is a simplicial manifold $N\widehat{G}$ and face operators $ \hat{\varepsilon}_i $
of it. Using this, Behrend and Xu described the cocycle which represents the DD class in another way.

\begin{proposition}[\cite{Beh}\cite{Beh2}]
{Let} $\widehat{G} \times \widehat{G} \rightarrow G \times G$ be a $(U(1) \times U(1))$-bundle.
Then the $1$-form $(\hat{{\varepsilon}}_{0} ^* - \hat{{\varepsilon}}_{1} ^* + \hat{{\varepsilon}}_{2} ^* )\theta$ on $\widehat{G} \times \widehat{G}$ is horizontal and $(U(1) \times U(1))$-invariant, hence  there exists the $1$-form $\chi$ on
${G} \times {G}$
 which satisfies
$(\pi \times \pi) ^* \chi = (\hat{{\varepsilon}}_{0} ^* - \hat{{\varepsilon}}_{1} ^* + \hat{{\varepsilon}}_{2} ^* )\theta$.

\end{proposition}
\begin{proof}
For example, see \cite[G.Ginot, M.Sti\'{e}non]{Gin}.
\end{proof}
\par

Behrend and Xu proved the theorem below in \cite{Beh2}.
\begin{theorem}[\cite{Beh}\cite{Beh2}]
{The cohomology class} $[c_1 (\theta ) - \left( \frac{-1}{2 \pi i} \right)\chi] \in H^3 (\Omega  (NG))$ 
represents the universal Dixmier-Douady class.
\end{theorem}

Now we show our cocycle in section 2.2 satisfies the required condition in Proposition 3.1 
when we choose a natural section $s_{nt}:G \times G \rightarrow \delta \widehat{G}$.

\begin{theorem}
{The equation} $(\pi \times \pi)^* s^* _{nt}(\delta \theta) = (\hat{{\varepsilon}}_{0} ^* - \hat{{\varepsilon}}_{1} ^* + \hat{{\varepsilon}}_{2} ^* )\theta$ holds.
\end{theorem}
\begin{proof}
Choose an open cover $\mathcal{V} = \{ V_{\lambda} \}_{\lambda \in \Lambda}$ of $G$ such that all the intersections of open sets in  $\mathcal{V}$ are contractible and
such that there exist local sections $ {\eta}_{\lambda} : V_{\lambda} \rightarrow \widehat{G}$ of $\pi$.
Then  $ \{ {\varepsilon}_{0} ^{-1} (V_{\lambda} ) \cap {\varepsilon}_{1} ^{-1} (V_{\lambda '} ) \cap {\varepsilon}_{2} ^{-1} (V_{\lambda ''}) \}_{\lambda , {\lambda}' , {\lambda}'' \in \Lambda}  $ is an open covering of $G \times G$ and there are the induced local
sections ${\varepsilon}_{0} ^* \eta _{\lambda} \otimes ({\varepsilon}_{1} ^* \eta _{\lambda '})^{\otimes -1} \otimes {\varepsilon}_{2} ^* \eta _{\lambda ''} $ on that covering.

If we pull back $\delta \theta $ by these sections, the induced form on ${\varepsilon}_{0} ^{-1} (V_{\lambda} ) \cap {\varepsilon}_{1} ^{-1} (V_{\lambda '} ) \cap {\varepsilon}_{2} ^{-1} (V_{\lambda ''})$ is ${\varepsilon}_{0} ^* (\eta _{\lambda} ^* \theta )- {\varepsilon}_{1} ^* (\eta _{\lambda '} ^* \theta )+ {\varepsilon}_{2} ^* ( \eta _{\lambda ''} ^* \theta )$.

We define the $U(1)$-valued functions $\tau_{\lambda \lambda' \lambda''}$ on  ${\varepsilon}_{0} ^{-1} (V_{\lambda} ) \cap {\varepsilon}_{1} ^{-1} (V_{\lambda '} ) \cap {\varepsilon}_{2} ^{-1} (V_{\lambda ''})$
as $ ({\varepsilon}_{0} ^* \eta _{\lambda} \otimes ({\varepsilon}_{1} ^* \eta _{\lambda '})^{\otimes -1} \otimes {\varepsilon}_{2} ^* \eta _{\lambda ''} ) \cdot \tau_{\lambda \lambda' \lambda''}= s_{nt} $.

Then  ${\varepsilon}_{0} ^* (\eta _{\lambda} ^* \theta )- {\varepsilon}_{1} ^* (\eta _{\lambda '} ^* \theta )+ {\varepsilon}_{2} ^* ( \eta _{\lambda ''} ^* \theta )+ \tau_{\lambda \lambda' \lambda''} ^{-1} d \tau_{\lambda \lambda' \lambda''}$ is equal to $s^* _{nt}\delta \theta $ hence $(\pi \times \pi)^*s^* _{nt} \delta \theta = (\pi \times \pi)^*({\varepsilon}_{0} ^* (\eta _{\lambda} ^* \theta )- {\varepsilon}_{1} ^* (\eta _{\lambda '} ^* \theta )+ {\varepsilon}_{2} ^* ( \eta _{\lambda ''} ^* \theta ))+ (\pi \times \pi)^*\tau_{\lambda \lambda' \lambda''} ^{-1} d \tau_{\lambda \lambda' \lambda''} $.

Let $\tilde{{\varphi}}_{\lambda}:{\pi }^{-1}(V_{\lambda}) \rightarrow V_{\lambda} \times U(1) $ be
a local trivialization of $\pi$ and we define ${\varphi}_{\lambda} :={ \rm{pr}}_2 \circ \tilde{{\varphi}}_{\lambda}:\pi ^{-1}(V_{\lambda}) \rightarrow U(1) $.
For any  $\hat{g} \in \pi ^{-1}(V_{\lambda})$ the equation $\hat{g}=\eta_{\lambda} \circ \pi (\hat{g}) \cdot {\varphi}_{\lambda}(\hat{g})$ holds
so we can see
$\hat{{\varepsilon}}_{i} ^* \theta =  {\hat{\varepsilon}}_{i} ^*(\pi ^*(\eta_{\lambda } ^*\theta))
+ \hat{{\varepsilon}}_{i} ^*{\varphi}_{\lambda } ^{-1} d {\varphi}_{\lambda} = (\pi \times \pi)^* {\varepsilon}_{i} ^*(\eta_{\lambda } ^*\theta)
+ \hat{{\varepsilon}}_{i} ^*{\varphi}_{\lambda } ^{-1} d {\varphi}_{\lambda}$ on  $\hat{{\varepsilon}}_{i} ^{-1}(\pi ^{-1}(V_{\lambda})) = (\pi \times \pi)^{-1}(\varepsilon_{i} ^{-1}(V_{\lambda}))$.

Therefore on $(\pi \times \pi)^{-1}(\varepsilon_{0} ^{-1}(V_{\lambda}) \cap \varepsilon_{1} ^{-1}(V_{\lambda '}) \cap \varepsilon_{2} ^{-1}(V_{\lambda ''}))$ there is a differential form $\hat{{\varepsilon}}_{0} ^* \theta
- \hat{{\varepsilon}}_{1} ^* \theta + \hat{{\varepsilon}}_{2} ^* \theta = (\pi \times \pi)^* ({\varepsilon}_{0} ^*(\eta_{\lambda } ^*\theta) - {\varepsilon}_{1} ^*(\eta_{\lambda '} ^*\theta)
+ {\varepsilon}_{2} ^*(\eta_{\lambda ''} ^*\theta))
+ \hat{{\varepsilon}}_{0} ^*{\varphi}_{\lambda } ^{-1} d {\varphi}_{\lambda} 
- \hat{{\varepsilon}}_{1} ^*{\varphi}_{\lambda '} ^{-1} d {\varphi}_{\lambda '} 
+ \hat{{\varepsilon}}_{2} ^*{\varphi}_{\lambda ''} ^{-1} d {\varphi}_{\lambda ''}$.

Since $ \hat{{\varepsilon}}_{i} = (\eta_{\lambda} \circ \pi \circ \hat{{\varepsilon}}_{i}) \cdot {\varphi}_{\lambda} \circ \hat{{\varepsilon}}_{i} = (\eta_{\lambda} \circ {\varepsilon}_{i} \circ (\pi \times \pi))
 \cdot {\varphi}_{\lambda} \circ \hat{{\varepsilon}}_{i}$, we can see that $\hat{{\varepsilon}}_{0}  \otimes
\hat{{\varepsilon}}_{1} ^{\otimes -1} \otimes \hat{{\varepsilon}}_{2} : \widehat{G} \times \widehat{G}
\rightarrow \delta \widehat{G} $ is equal to $(({\varepsilon}_{0} ^* \eta _{\lambda} \otimes ({\varepsilon}_{1} ^* \eta _{\lambda '})^{\otimes -1} \otimes {\varepsilon}_{2} ^* \eta _{\lambda ''} ) \circ (\pi \times \pi))
\cdot ({\varphi}_{\lambda} \circ \hat{{\varepsilon}}_{0})({\varphi}_{\lambda '} \circ \hat{{\varepsilon}}_{1})^{-1}({\varphi}_{\lambda ''} \circ \hat{{\varepsilon}}_{2})$.

We have $ \tau_{\lambda \lambda' \lambda''} \circ (\pi \times \pi) = ({\varphi}_{\lambda} \circ \hat{{\varepsilon}}_{0})({\varphi}_{\lambda '} \circ \hat{{\varepsilon}}_{1})^{-1}({\varphi}_{\lambda ''} \circ \hat{{\varepsilon}}_{2})$ because
$s _{nt}\circ (\pi \times \pi) = \hat{{\varepsilon}}_{0}  \otimes \hat{{\varepsilon}}_{1} ^{\otimes -1} \otimes \hat{{\varepsilon}}_{2}$, so it follows that $(\hat{{\varepsilon}}_{0} ^* - \hat{{\varepsilon}}_{1} ^* + \hat{{\varepsilon}}_{2} ^* )\theta = (\pi \times \pi)^*s^* _{nt}\delta \theta$. 
This completes the proof.

\end{proof}

\section{The String class}
Using the results of Brylinski, McLaughlin \cite{Bry-ML} and Murray, Stevenson \cite{Mur}\cite{Mur2}, we discuss the case of central $U(1)$-extensions of the free loop groups.

\subsection{In the case of special unitary group}

It's known that the second Chern class $c_2 \in H^4(BSU(2))$ of the universal $SU(2)$-bundle $ESU(2) \rightarrow BSU(2)$ is represented in $\Omega ^4(NSU(2))$
as the sum of following $C_{1,3}$ and $C_{2,2}$ (see for example \cite{Jef} or \cite{Suz3}):
$$
\begin{CD}
0 \\
@AA{-d}A \\
C_{1,3} \in {\Omega}^{3} (SU(2) )@>{{\varepsilon}_{0} ^* - {\varepsilon}_{1} ^* + {\varepsilon}_{2} ^*}>>{\Omega}^{3} (SU(2) \times SU(2))\\
@.@AA{d}A\\
@.C_{2,2} \in {\Omega}^{2} (SU(2) \times SU(2))@>{{\varepsilon}_{0} ^* - {\varepsilon}_{1} ^* + {\varepsilon}_{2} ^* - {\varepsilon}_{3} ^*}>> 0
\end{CD}
$$
$$C_{1,3} = \left( \frac{1}{2 \pi i} \right) ^2  \frac{-1}{6}{\rm tr}({h^{-1}dh} )^3 ,
\qquad C_{2,2} = \left( \frac{1}{2 \pi i} \right) ^2  \frac{1}{2}{\rm tr}(h_2 ^{-1} h_1 ^{-1} dh_1 dh_2  ).$$

Pulling back this cocycle by the evaluation map $ev:LSU(2) \times S^1 \rightarrow SU(2) , (\gamma, z) \mapsto \gamma (z)$ and integrating it along the circle,
we obtain the cocycle in $\Omega ^3(NLSU(2))$.
Here $LSU(2)$ is the free loop group of $SU(2)$ and the map $\int _{S^1} ev^*$ is called the transgression map.\\

Now we pose the following problem. ``Is there corresponding central extension of $LSU(2)$ and connection form on it 
such that the DD class in $\Omega ^3(NLSU(2))$ constructed previous section coincides with $\int _{S^1} ev ^* (C_{1,3} +C_{2,2})$?''
In this section, we explain that the central extension and the connection form constructed by Mickelsson \cite{Mic} and Brylinski, McLaughlin \cite{Bry} \cite{Bry-ML}  meet such a condition.\\

To begin with, we recall the definition of the $U(1)$-bundle $\pi : Q(\nu) \rightarrow LSU(2)$ and the multiplication  $m:Q(\nu) \times Q(\nu) \rightarrow Q(\nu)$ in \cite{Bry} \cite{Bry-ML}. We fix any based point $x_0 \in SU(2)$
and denote $\gamma _0 \in LSU(2)$ the constant loop at $x_0$. For any $\gamma \in LSU(2)$, we consider all paths $\sigma _{\gamma} :[0,1] \rightarrow
LSU(2)$ that satisfies $\sigma _{\gamma}(0)= \gamma _0$ and $\sigma _{\gamma}(1)= \gamma$. Then the equivalence relation $\sim$
on  $\{\sigma _{\gamma} \} \times S^1$ is defined as follows:
$$(\sigma _{\gamma},z) \sim (\sigma '_{\gamma},z') \Leftrightarrow z = z' \cdot {\rm{exp}} \left( \int_{I^2 \times S^1} 2 \pi i F^* \nu \right).$$
Here $F:I^2 \times S^1 \rightarrow SU(2)$ is any homotopy map that satisfies $F(0,t,z)=\sigma _{\gamma}(t)(z)$, $F(1,t,z)=\sigma ' _{\gamma}(t)(z)$
and $\nu$ is $C_{1,3} = \left( \frac{-1}{2 \pi i} \right) ^2  \frac{-1}{6}{\rm tr}({h^{-1}dh} )^3$. It's well known $\nu \in \Omega ^3(SU(2))$ is a closed, integral form hence this relation is well-defined.
Now the fiber $\pi ^{-1}(\gamma)$ of $Q(\nu)$
is defined as the quotient space $\{\sigma _{\gamma}\} \times S^1/ \sim $. 

We can adapt the same construction for any closed integral $3$-form on $SU(2)$.
Let $\eta$,$\eta'$ be such $3$-forms and suppose there is a $2$-form $\beta$ with $d\beta = \eta' -\eta$. Then 
the isomorphism from $Q(\eta)$ to $Q(\eta')$ is constructed as:
$$[(\sigma _{\gamma},z)]_{\eta} \mapsto [(\sigma _{\gamma},z \cdot {\rm{exp}} \left(  \int_{I^1 \times S^1} 2 \pi i {\sigma}_{\gamma} ^* \beta \right))]_{\eta'}.$$
Here we regard $\sigma _{\gamma}$ as a map from $[0,1] \times S^1$ to $SU(2)$.

For the face operators $\{ {\varepsilon}_{i}  \}:SU(2) \times SU(2) \rightarrow SU(2)$ (we use the same notation for the face operators $LSU(2) \times LSU(2) \rightarrow LSU(2)$),
we can check ${\varepsilon}_{0} ^* Q(\nu) \otimes {\varepsilon}_{1} ^* Q(\nu)^{\otimes -1} \otimes {\varepsilon}_{2} ^* Q(\nu)$ is isomorphic to $Q({\varepsilon}_{0} ^* \nu -{\varepsilon}_{1} ^* \nu +
{\varepsilon}_{2} ^* \nu) = Q(-dC_{2,2})$ over $LSU(2) \times LSU(2)$. The isomorphism from $Q(0)$ to $Q(-dC_{2,2})$ is
given by 
$$[(\sigma _{\gamma _1},\sigma _{\gamma _2},z)]_{0} \mapsto [(\sigma _{\gamma _1},\sigma _{\gamma _2}, z \cdot {\rm{exp}} \left(  \int_{I^1 \times S^1} 2 \pi i ({\sigma}_{\gamma _1},\sigma _{\gamma _2}) ^* C_{2,2} \right))]_{-dC_{2,2}}.$$
Now we can define the section $s_L$ of ${\varepsilon}_{0} ^* Q(\nu) \otimes {\varepsilon}_{1} ^* Q(\nu)^{\otimes -1} \otimes {\varepsilon}_{2} ^* Q(\nu)
$ over $LSU(2) \times LSU(2)$
as:
$$s_L(\gamma _1, \gamma_2 ) := [(\sigma _{\gamma _1},\sigma _{\gamma _2}, {\rm{exp}} \left(  \int_{I^1 \times S^1} 2 \pi i ({\sigma}_{\gamma _1},\sigma _{\gamma _2}) ^* C_{2,2} \right))]_{-dC_{2,2}}.$$
The multiplication  $m:Q(\nu) \times Q(\nu) \rightarrow Q(\nu)$ is defined by the following equation
$$s_L(\gamma_1,\gamma_2)=([\sigma _{\gamma _1},z_1]_{\varepsilon_0 ^* \nu})\otimes((\gamma_1 \gamma_2),m([\sigma _{\gamma _1},z_1]_{ \nu},
[\sigma _{\gamma _2},z_2]_{ \nu}))^{\otimes -1}\otimes([\sigma _{\gamma _2},z_2]_{\varepsilon_2 ^* \nu}).$$

Next we recall how Brylinski and McLaughlin constructed the connection on $Q(\nu)$. Let denote $P_1SU(2)$ the space of paths on $SU(2)$ which starts 
from based point $x_0$ and $f:P_1SU(2) \rightarrow SU(2)$ a map that is defined by $f(\gamma)=\gamma(1)$.
It is well known that $f$ is a fibration. Then we define the $2$-form $\omega$ on $P_1SU(2)$
as:
$$\omega_{\gamma}(u,v)=\int ^1 _0 \nu \left(  \frac{d \gamma}{dt} ,u(t),v(t) \right)dt.$$
Note that $d \omega = f^* \nu $ holds. Let $\mathcal{U}=\{U_{\iota} \}$ be an open covering of $SU(2)$. Since $SU(2)$ is simply connected, we can take $\mathcal{U}$ such that
each $U_{\iota}$ is contractible and $\{LU_{\iota} \}$ is an open covering of $LSU(2)$. For example, we take $\mathcal{U}=\{U_{x}:= SU(2)-\{x\} |x \in SU(2) \}$.

Now we quote the lemma from \cite{Bry-ML}.
\begin{lemma}[Brylinski, McLaughlin \cite{Bry-ML}]
{(1) There exists a line bundle $L$ over each} $f^{-1}(U_{\iota})$ with a fiberwise connection such that its first Chern form is equal to $ \omega|_{f^{-1}(U_{\iota})}$.
This line bundle is called the pseudo-line bundle.\\
(2) There exists a connection $\nabla$ on each pseudo-line bundle $L$ such that its first Chern form $R$ satisfies the condition
that $R- \omega|_{f^{-1}(U_{\iota})}$ is basic.

\end{lemma}
Let $K$ be a $2$-form on $U_{\iota}$ which satisfies $f^*K = 2 \pi i ( R- \omega|_{f^{-1}(U_{\iota})})$. Then the $1$-form $\theta_{\iota}$ on $LU_{\iota}$ 
is defined by $\theta_{\iota} := \int _{S^1} ev ^*K$. It is easy to see $ \left( \frac{-1}{2 \pi i} \right)d \theta_{\iota} = (\int_{S^1} ev^* \nu) |_{LU_{\iota}}$.\\

There is a section $s_{\iota}$ on $LU_{\iota}$ defined by $s_{\iota}(\gamma) := [\sigma _{\gamma }, H_{\sigma _{\gamma }}(L,\nabla)]$.
Here $H_{\sigma _{\gamma }}(L,\nabla)$ is the holonomy of $(L,\nabla)$ along the loop $\sigma_{\gamma}:S^1 \rightarrow f^{-1}(U_{\iota})$.
We also have the corresponding local trivialization ${\varphi}_{\iota}:\pi ^{-1}(U_{\iota}) \rightarrow U_{\iota} \times U(1)$.\\

From above, we have the connection form $\theta$ on $Q(\nu)$ defined by $\theta |_{\pi ^{-1}(U_{\iota})} := \pi ^* \theta_{\iota} + d {\rm{log}}({\rm{pr_2}} \circ {\varphi}_{\iota})$. Its first Chern form $c_1(\theta)$ is $\int_{S^1} ev^* \nu$ and $d \delta \theta $ is equal to $(-2 \pi i)\cdot \int _{S^1}ev^*(({\varepsilon}_{0} ^* - {\varepsilon}_{1} ^* + {\varepsilon}_{2} ^*) \nu) = (-2 \pi i)\cdot \pi ^* \left( - d \int _{S^1}ev^*C_{2,2} \right)$
hence $\delta \theta + (-2 \pi i)\cdot \pi^* \int _{S^1}ev^*C_{2,2} $ is a flat connection on $\delta Q(\nu)$. Since $LSU(2)$ is simply connected,
it is a trivial connection so $s_L ^* ( \delta \theta + (-2 \pi i)\cdot \pi^* \int _{S^1}ev^*C_{2,2}) = 0$.\\

So as a reformulation of the Brylinski and McLaughlin's result, we obtain the proposition below.
\begin{proposition}[\cite{Bry-ML}]
{Let} $( Q(\nu),\theta)$ be a $U(1)$-bundle on $LSU(2)$ with connection and $s_L$ be a global section of $\delta Q(\nu)$ constructed above.
Then the cocycle $c_1(\theta) - \left( \frac{-1}{2 \pi i} \right){s_L}^{*} (\delta \theta )$ on $\Omega ^3(NLSU(2))$ is
equal to $\int _{S^1} ev ^* (C_{1,3} +C_{2,2})$, i.e. the map $\int _{S^1} ev ^*$ sends the second Chern class $c_2 \in H^4(BSU(2))$ to the Dixmier-Douady class
(associated to $Q(\nu)$) in $H^3(BLSU(2))$.\hspace{9em} $ \Box $
\end{proposition}

\bigskip

\subsection{In the case of unitary group}

In the case of unitary group $U(2)$, the second Chern class is represented as the sum of following $C^U _{1,3}$ and $C^U _{2,2}$ (see \cite{Suz3}):

$$
\begin{CD}
0 \\
@AA{-d}A \\
C^U _{1,3} \in {\Omega}^{3} (U(2) )@>{{\varepsilon}_{0} ^* - {\varepsilon}_{1} ^* + {\varepsilon}_{2} ^*}>>{\Omega}^{3} (U(2) \times U(2))\\
@.@AA{d}A\\
@.C^U _{2,2} \in {\Omega}^{2} (U(2) \times U(2))@>{{\varepsilon}_{0} ^* - {\varepsilon}_{1} ^* + {\varepsilon}_{2} ^* - {\varepsilon}_{3} ^*}>> 0
\end{CD}
$$
$$C^U _{1,3} = \left( \frac{1}{2 \pi i} \right) ^2  \frac{-1}{6}{\rm tr}({h^{-1}dh} )^3 \hspace{17em} $$
$$C^U _{2,2} = \left( \frac{1}{2 \pi i} \right) ^2  \frac{1}{2}{\rm tr}(h_2 ^{-1} h_1 ^{-1} dh_1 dh_2  )- \left( \frac{1}{2 \pi i} \right) ^2  \frac{1}{2}{\rm tr}( h_1 ^{-1} dh_1){\rm tr}( h_2 ^{-1}dh_2  ).$$

\bigskip
Recall that any element of $U(2)$ can be represented as follows:
$$ U(2) = \bigg\{ 
\left (\begin{array}{cc}
   \ \alpha & - z \overline{\beta} \\
   \  \beta & z \overline{\alpha} \\
\end{array} \right) 
: \alpha,\beta,z \in \mathbb{C} , |z|=1, |\alpha|^2 +|\beta|^2 =1 \bigg\}.
$$

We recognize $U(2)$ as a semi-direct product group $SU(2) \rtimes U(1)$. Here we define the product $\rtimes$ as:
$$ (A_1,z_1) \rtimes (A_2,z_2) := (A_1
\left (\begin{array}{cc}
   \ 1 & 0 \\
   \ 0 & z_1 \\
\end{array} \right) A_2
\left (\begin{array}{cc}
   \ 1 & 0 \\
   \ 0 & z_1 \\
\end{array} \right)^{-1},z_1z_2).
$$

Let denote $\Omega U(1)$ the based loop group of $U(1)$. Then any element $\gamma$ in $LU(2)$ is decomposed as $\gamma = (\gamma_1, \gamma_2, z) \in LSU(2) \rtimes (\Omega U(1) \rtimes U(1))$. Each connected component of $LU(2)$ is parametrized by the mapping degree of $\gamma _2$.
We write $\Omega U(1)_n$,$LU(2)_n$ the connected components which include a based loop $\gamma_2$ whose mapping degree is $n$.
We can see $\pi_1(LU(2)_0) = \pi_1(LSU(2)) \oplus \pi_1(LU(1)_0) = \pi_1(LU(1)_0) = \pi_1(\Omega U(1)_0) \oplus \pi_1(U(1)) \cong \mathbb{Z}$.
There is a homeomorphism from $\Omega U(1)_0$ to $\Omega U(1)_n$ defined by $\gamma \mapsto \gamma \cdot (e^{is} \mapsto e^{ins})$ for any $n$
so $\pi_1(LU(2)_n)$ is also isomorphic to $\mathbb{Z}$.
The generator $\psi_n$ of $\pi_1(LU(2)_n) \cong H_1(LU(2)_n)$ is the map defined as $\psi_n(e^{it}) := (e^{is} \mapsto \binom{1 \quad 0}{0 e^{i(ns+t)}})$ hence any cycle $a \in Z_1(LU(2)_n)$ can be written as $a=m\psi_n + \partial \varrho $ for some $2$-chain $\varrho$
and $m \in {\mathbb Z}$.\\

Since $LU(2)$ is not simply connected we need the differential character $k$ to construct a principal $U(1)$-bundle over $LU(2)$.
Differential character is a homomorphism from $Z_1(LU(2))$ to $U(1)$ such that there exists a specific $2$-form $\omega$
satisfying $k(\partial \varrho) = {\rm{exp}}( \int_{\varrho} 2 \pi i\omega )$ for any $2$-singular chains ${\varrho}$ of $LU(2)$ (\cite{Chee} see also \cite{Mur2}).

We set $\Phi _{1,3}:= \int_{S^1} ev^* C^U _{1,3}$ and  define $k$ as $k(a) :={\rm{exp}}( \int_{\varrho} 2 \pi i \Phi_{1,3})$. This is well-defined
since  $\Phi_{1,3}$ is integral.

Now the equivalence relation $\sim$
on  $\{\sigma _{\gamma} \} \times S^1$ is defined as follows:
$$(\sigma _{\gamma},z) \sim (\sigma '_{\gamma},z') \Leftrightarrow z = z' \cdot   k(\sigma ^{-1}_{\gamma} \circ\sigma '_{\gamma}).$$

Then we obtain a $U(1)$-bundle $Q(k)$ over $LU(2)$. We can check ${\varepsilon}_{0} ^* Q(k) \otimes {\varepsilon}_{1} ^* {Q}(k)^{\otimes -1} \otimes {\varepsilon}_{2} ^* Q(k)$ is isomorphic to $Q(({\varepsilon}_{0} ^*  -{\varepsilon}_{1} ^*  +
{\varepsilon}_{2} ^* )k)$ over $LU(2) \times LU(2)$.

We set $\Phi _{2,2}:= -\int_{S^1} ev^* C^U _{2,2}$. Since  $\Phi_{2,2}$ is integral, the following equation holds for any $(a_1,a_2)= (m_1 \psi^1_{n_1} + \partial \varrho_1, m_2 \psi^2_{n_2} + \partial \varrho_2) \in Z_1(LU(2) \times LU(2))$:
$$(({\varepsilon}_{0} ^*  -{\varepsilon}_{1} ^*  +
{\varepsilon}_{2} ^* )k)(a_1,a_2)=k(({\varepsilon_0}_*-{\varepsilon_1}_*+{\varepsilon_2}_*)(a_1,a_2))$$
$$={\rm exp} \int_{({\varepsilon_0} _*-{\varepsilon_1}_*+{\varepsilon_2}_*)(\varrho_1,\varrho_2)}2 \pi i\Phi_{1,3}$$
$$={\rm exp} \int_{(\varrho_1,\varrho_2)}2 \pi i(\varepsilon_0 ^*-\varepsilon_1 ^*+\varepsilon_2 ^*)\Phi_{1,3}={\rm exp} \int_{(\varrho_1,\varrho_2)}2 \pi i (d\Phi _{2,2})$$
$$={\rm exp} \int_{(\partial\varrho_1,\partial\varrho_2)}2 \pi i\Phi _{2,2}= {\rm exp} \int_{(a_1,a_2)} 2 \pi i\Phi _{2,2}=1.$$

Therefore ${\varepsilon}_{0} ^* Q(k) \otimes {\varepsilon}_{1} ^* {Q}(k)^{\otimes -1} \otimes {\varepsilon}_{2} ^* Q(k)$ is 
a trivial bundle, so repeating the same argument in section 4.1, we obtain a central $U(1)$-extension of $LU(2)$.

\section{Cocycle in the triple complex}
In this section we deal with a semi-direct product $LG \rtimes S^1$ for $G=SU(2)$. Here we define a group action of $S^1$
on $LG$ by the adjoint action of $S^1$ on $SU(2)$, i.e. we fix a semi-direct product operator $\cdot_{\rtimes}$of $LG \rtimes S^1$ as $(\gamma, z) \cdot_{\rtimes} ({\gamma}', z') := ({\gamma}\cdot z{\gamma}'z^{-1}, zz')$. So in this case $LG \rtimes S^1$ is a subgroup of $LU(2)$.

First we define bisimplicial manifolds $NLG(*) \rtimes NS^1(*)$ and $PLG(*) \times PS^1(*)$. A bisimplicial manifold is a sequence of manifolds with horizontal and vertical face and degeneracy operators which commute with each other.
A bisimplicial map is a sequence of maps commuting with horizontal and vertical face and degeneracy operators.
We define $NLG(*) \rtimes NS^1(*)$ as follows:
$$NLG(p) \rtimes NS^1(q)  := \overbrace{LG \times \cdots \times LG }^{p-times} \times \overbrace{S^1 \times \cdots \times S^1 }^{q-times} $$  
Horizontal face operators \enspace ${\varepsilon}_{i}^{LG} : NLG(p) \rtimes NS^1(q) \rightarrow NLG(p-1)  \rtimes NS^1(q) $ are the same with the face operators of $NLG(p)$. Vertical face operators \enspace ${\varepsilon}_{i}^{S^1} : NLG(p) \rtimes NS^1(q) \rightarrow NLG(p)  \rtimes NS^1(q-1) $ are
$$
{\varepsilon}_{i}^{S^1}(\vec{\gamma}, z_1 , \cdots , z_q )=\begin{cases}
(\vec{\gamma}, z_2 , \cdots , z_q )  &  i=0 \\
(\vec{\gamma}, z_1 , \cdots ,z_i z_{i+1} , \cdots , z_q )  &  i=1 , \cdots , q-1 \\
(z_q \vec{\gamma}z^{-1} _{q}, z_1 , \cdots , z_{q-1} )  &  i=q
\end{cases}
$$
Here $\vec{\gamma}=(\gamma_1, \cdots , \gamma_p)$. $PLG(*) \times PS^1(*)$ is defined as:  
$$PLG(p) \times PS^1(q)  := \overbrace{LG \times \cdots \times LG }^{p+1-times} \times \overbrace{S^1 \times \cdots \times S^1 }^{q+1-times}. $$
Its $i$-th face operator is the map which omits the $i$-th factor.
Then we define a bisimplicial map $\rho_{\rtimes} : P{LG}(p) \times P{S^1}(q) \rightarrow NLG(p) \rtimes NS^1(q) $ as $ \rho_{\rtimes} (\vec{\gamma}, z_1, \cdots ,z_{q+1} ) = ( z _{q+1}\rho (\vec{\gamma}) z^{-1} _{q+1} ,
 \rho (z_1, \cdots, z_{q+1}))$ where $\rho$ is defined as $\rho (\vec{\gamma}) := (\gamma_1 \gamma_2 ^{-1}, \cdots , \gamma_p \gamma_{p+1} ^{-1})$.

$LG \rtimes S^1$ acts on $ P{LG}(p) \times P{S^1}(q)$ by right as $(\vec{\gamma},\vec{z})\cdot(\gamma,z) = (z^{-1}\vec{\gamma}\cdot {\gamma}z, \vec{z}z)$. Since $\rho_{\rtimes}(\vec{\gamma},\vec{z})=
\rho_{\rtimes}((\vec{\gamma},\vec{z})\cdot(\gamma,z))$, one can see that $\rho_{\rtimes}$ is a principal  $(LG \rtimes S^1)$-bundle. $\parallel P{LG}(*) \times P{S^1}(*) \parallel$ is a product space of total spaces of the universal $LG$-bundle and the universal $S^1$-bundle. We can check also that $\parallel P{LG}(*) \times P{S^1}(*) \parallel \rightarrow \parallel NLG(*) \rtimes NS^1(*) \parallel$ is a principal $LG \rtimes S^1$-bundle since $LG$ and $S^1$ are ANR. Hence $\parallel NLG(*) \rtimes NS^1(*) \parallel$ is a model of $B(LG \rtimes S^1)$.
 
\begin{definition}
\rm{For a bisimplicial manifold} $NLG(*) \rtimes NS^1(*)$, we have a triple complex as follows:
$${\Omega}^{p,q,r} (NLG(*) \rtimes NS^1(*)) := {\Omega}^{r} (NLG(p) \rtimes NS^1(q)). $$

Derivatives are:
$$ d' = \sum _{i=0} ^{p+1} (-1)^{i} ({{\varepsilon}^{LG} _{i}}) ^{*}  , \qquad  d'' = \sum _{i=0} ^{q+1} (-1)^{i} ({{\varepsilon}^{S^1} _{i}}) ^{*} \times (-1)^{p} $$
$$ d''' =  (-1)^{p+q} \times {\rm the \enspace exterior \enspace differential \enspace on \enspace }{ \Omega ^*(NLG(p) \rtimes NS^1(q)) }.$$

\end{definition}
$\hspace{30em} \Box $\\
The following proposition can be proved by adapting the same argument in the proof of Theorem 2.1 (See \cite{Suz2}).
\begin{proposition}
{There exists an isomorphism}
$$ H({\Omega}^{*} (NLG \rtimes NS^1))  \cong  H^{*} (B(LG \rtimes S^1) ).$$

{ Here} ${\Omega}^{*} (NLG \rtimes NS^1)$ { means the total complex}.
\end{proposition}
$\hspace{30em} \Box $\\
Now we want to construct a cocycle in ${\Omega}^{3} (NLG \rtimes NS^1)$ which coincides with $c_1(\theta) - \left( \frac{-1}{2 \pi i} \right){s_L}^{*} (\delta \theta )$  when it is restricted to ${\Omega}^{3} (NLG)$.

To do this, it suffice to construct a differential form $\tau$ in $\Omega^1(LG \rtimes S^1)$ such that
$d \tau = (-{{\varepsilon}_{0}^{S^1}}^*+{{\varepsilon}_{1}^{S^1}}^*)c_1(\theta)$ and $({{{\varepsilon}^{LG} _{0}}}^*-{{{\varepsilon}^{LG} _{1}}}^*
+{{{\varepsilon}^{LG} _{2}}}^*)\tau =({{\varepsilon}_{0}^{S^1}}^*-{{\varepsilon}_{1}^{S^1}}^*)\left( \frac{-1}{2 \pi i} \right){s_L}^{*} (\delta \theta )$ and 
$(-{{\varepsilon}_{0}^{S^1}}^*+{{\varepsilon}_{1}^{S^1}}^*- {{\varepsilon}_{2}^{S^1}}^*)\tau =0$. We consider a trivial
$U(1)$-bundle $({{\varepsilon}_{0}^{S^1}}^*Q)^{\otimes -1}\otimes {{\varepsilon}_{1}^{S^1}}^*Q$ and the induced connection 
form $\delta_{\rtimes} \theta$ on it. We define a section $s_{\rtimes}:LG \rtimes S^1 \rightarrow ({{\varepsilon}_{0}^{S^1}}^*Q)^{\otimes -1}\otimes {{\varepsilon}_{1}^{S^1}}^*Q$ as $s_{\rtimes}(\gamma,z):=(\hat{\gamma},z)^{\otimes -1}\otimes(z^{-1} \hat{\gamma}z,z)$ and set $\tau :=
\left( \frac{-1}{2 \pi i} \right){s_{\rtimes}}^{*} (\delta _{\rtimes}\theta )$ then we can see
that  $\tau$ satisfies the required conditions.

\end{document}